\documentclass[12pt]{amsart}

\usepackage{amssymb}
\usepackage{bm}
\usepackage{graphicx}
\usepackage{psfrag}
\usepackage[dvipsnames]{xcolor}
\usepackage{enumerate}
\usepackage{multirow}
\usepackage{url}
\usepackage{mathpazo}

\usepackage{subcaption}

\usepackage{comment}

\usepackage{algpseudocode}

\usepackage{mathtools}

\usepackage{xy}
\input xy
\xyoption{all}

\usepackage{tikz}

\newcommand{\excise}[1]{}

\newtheorem{thm}{Theorem}[section]
\newtheorem{lemma}[thm]{Lemma}

\newtheorem{conj}[thm]{Conjecture}

\theoremstyle{definition}

\newtheorem{example}[thm]{Example}
\newtheorem{remark}[thm]{Remark}

\numberwithin{equation}{section}

\renewcommand\>{\rangle}
\newcommand\<{\langle}

\newcommand\ZZ{\mathbb{Z}}

\DeclareMathOperator\Ap{Ap} 

\makeatletter
\def\@bignumber#1#2{%
  \ifx#2\end
    #1\let\next\@gobble
  \else
    #1\hspace{0pt plus 1pt}\let\next\@bignumber
  \fi
  \next#2}
\newcommand{\bignumber}[1]{\@bignumber#1\end}
\makeatother

\begin{document}

\mbox{}
\title[Lengths of irreducible decompositions of numerical semigroups]{Lengths of irreducible decompositions \\ of numerical semigroups}

\author[Garc\'ia-S\'anchez]{Pedro A. Garc\'ia-S\'anchez}
\thanks{The first author acknowledges financial support from grant
  PID2022-138906NB-C21, funded by
  MICIU/AEI/\bignumber{10.13039}/\bignumber{501100011033} and by ERDF
  ``A way of making Europe''; and from the Spanish Ministry of Science
  and Innovation (MICINN), through the ``Severo Ochoa and Mar\'{\i}a de
  Maeztu Programme for Centres and Unities of Excellence''
  (CEX2020-001105-M).}
\address{Departamento de Álgebra and IMAG, Universidad de Granada, Granada, Spain}
\email{pedro@ugr.es}

\author[O'Neill]{Christopher O'Neill}
\thanks{Much of the work in this manuscript was completed on the second author's visit to Universidad de Granada in 2025, funded by MICIU/AEI/\bignumber{10.13039}/\bignumber{501100011033}, grant PID2022-138906NB-C21.}
\address{Mathematics Department\\San Diego State University\\San Diego, CA 92182}
\email{cdoneill@sdsu.edu}

\date{\today}

\begin{abstract}
A numerical semigroup is an additive subsemigroup of the natural numbers that contains zero and has finite complement. A numerical semigroup is irreducible if it cannot be written as an intersection of numerical semigroups properly containing it.  It is known that every numerical semigroup can be decomposed as an intersection of irreducible numerical semigroups, but there can be multiple such decompositions, even when irredundancy is required. In this paper, we study the set of all decomposition lengths of a given numerical semigroup. It is conjectured that the set of decomposition lengths is always an interval; we prove this conjecture for numerical semigroups whose smallest positive element is at most six. Additionally, we examine a class of numerical semigroups that was recently shown to achieve arbitrarily large minimum decomposition length, and construct a family of irreducible decompositions whose lengths form a large interval.
\end{abstract}

\keywords{Numerical semigroup, irreducible numerical semigroup, decomposition into irreducibles, Apéry set, special gap, ordinary numerical semigroup}
\subjclass[2020]{20M14, 20M10, 20M13}

\maketitle


\section{Introduction}
\label{sec:intro}

A \emph{numerical semigroup} is a set of non-negative integers closed under addition and with finite complement in the set of non-negative integers. The set of integers not belonging to a numerical semigroup $S$ is called the set of \emph{gaps} of $S$; the cardinality of the sets of gaps is called the \emph{genus} of $S$ and is denoted by $\operatorname{g}(S)$. The largest integer not belonging to $S$ is known as its \emph{Frobenius number}, and it is denoted by $\operatorname{F}(S)$. The smallest positive integer in $S$ is called the \emph{multiplicity} of $S$, and denoted $\operatorname{m}(S)$.

A numerical semigroup is \emph{irreducible} if it cannot be expressed as the intersection of two numerical semigroups properly containing it. There are many characterizations of irreducible numerical semigroups; see \cite[Chapter~3]{numerical} for an overview. A~numerical semigroup $S$ with Frobenius number $F$ is irreducible if and only if it is maximal (under inclusion) among numerical semigroups not containing $F$.
If $F$ is odd, then $S$ is irreducible if and only if $\operatorname{g}(S) = \tfrac{1}{2}(F + 1)$, and in this case $S$ is called \emph{symmetric} (for every integer $x$, we have $x \notin S$, if and only if $F-x\in S$). If $F$ is even, then $S$ is irreducible if and only if $\operatorname{g}(S) = \tfrac{1}{2}F + 1$, and in this case $S$ is called \emph{pseudo-symmetric} (for every integer $x \ne \tfrac{1}{2}F$, we have $x\notin S$ if and only if $F-x\in S$).

Every numerical semigroup can be expressed as an intersection of finitely many irreducible numerical semigroups. If $S = S_1 \cap \cdots \cap S_k$ with each $S_i$ irreducible, then we call this expression a \emph{decomposition} of $S$ into irreducibles if $S \neq \cap_{i\in I} S_i$ for every proper subset $I \subsetneq \{1, \ldots, k\}$ (that is, any decomposition of $S$ into irreducibles is required to be \emph{irredundant}).  The number $k$ is called the \emph{length} of the decomposition.  

In this paper, we study the set of all decomposition lengths of a given numerical semigroup.  
We briefly compare to another form of decomposition:\ factorizations of semigroup elements as sums of minimal generators.  There is extensive literature studying lengths of such factorizations, both for numerical semigroups and more broadly, where a rich theory of combinatorial factorization invariants has been developed~\cite{gz}.  

In contrast, much less is known about decomposition lengths of numerical semigroups.  
Upper and lower bounds on the maximum and minumum decomposition lengths have been given in \cite{r-b}, but not until recently was it shown that the minimum decomposition length of a numerical semigroup can be arbitrarily large~\cite{b-f}, or that for every integer $k \ge 2$, there exists a numerical semigroup $S$ with decompositions of length 2 and $k$~\cite{gs}.  

Another way to view decomposition is as a form of ``factorization'' of numerical semigroups under the intersection operation.  Intersection is commutative and associative, but since it is far from being cancellative, much of the classical factorization theory does not immediately extend.  
The decompositions we consider here are aligned to the idea of minimal factorizations in this setting, as introduced in \cite{c-t}, though they already appear in \cite{a-t}.  In particular, since numerical semigroups are idempotent with respect to intersection, we necessarily impose non-redundancy as a requirement for the decompositions we consider.  


The following conjecture, which appeared in~\cite{gs}, provides a primary motivation for the present work.  

\begin{conj}
The set of decomposition lengths of any numerical semigroup is an interval.
\end{conj}

In Section~\ref{sec:low-multiplicity}, we prove this conjecture holds for numerical semigroups with multiplicity at most six.  To this end, we develop some tools involving Apéry sets and special gaps of numerical semigroups, which are presented in Section~\ref{sec:apery-and-special-gaps}.  We also examine in Section~\ref{sec:ordinary} a family of numerical semigroups shown in~\cite{b-f} to achieve arbitrarily large minimum decomposition length; for each semigroup in this family, we locate a large interval of decomposition lengths.  

The experiments that led to the results in this paper were carried out using the \texttt{GAP} \cite{gap} package \texttt{NumericalSgps} \cite{numericalsgps}.


\section{Apéry sets and special gaps}
\label{sec:apery-and-special-gaps}

In this section, we analyze the relevance of Apéry sets and special gaps in the study of decompositions of numerical semigroups into irreducibles. We will show how these concepts are interconnected and how they can be used to derive bounds on the lengths of decompositions.

For a fixed numerical semigroup $S$, we can define the divisibility relation on $\mathbb{Z}$: 
\[
a \preceq_S b \quad \text{whenever} \quad b - a \in S.
\]
This is a partial order on $\mathbb{Z}$. When the semigroup $S$ is clear from context, we sometimes simply write $\preceq$.  

Let $S$ be a numerical semigroup and let $n$ be a nonzero element of $S$. The \emph{Apéry set} of $S$ with respect to $n$ is defined as
\[
\Ap(S;n)=\{ s\in S : s - n \notin S\}.
\]
It is well known that $\Ap(S;n)$ has cardinality $n$ and that $\{0,a_1,\dots,a_{n-1}\}$, where each $a_i$ is the smallest element of $S$ congruent with $i$ modulo $n$. In particular, 
\begin{equation}\label{eq:membership-with-apery}
s\in S \text{ if and only if } a_i \le s.
\end{equation}
for each $s\in \mathbb{Z}$ with $s \equiv i \bmod n$.  This membership criterion is central to many results, in this paper and elsewhere.  When $n$ is the multiplicity of $S$ we will sometimes write $\Ap(S)$ instead of $\Ap(S;n)$ for the sake of simplicity.

It is well known that irreducible numerical semigroups can be characterized in terms of their Apéry sets.

\begin{lemma}[{\cite[Propositions~4.10 and~4.15]{numerical}}]\label{lem:pseudosymmposet}
Let $S$ be a numerical semigroup, and fix a nonzero $m \in S$.
Suppose $\Ap(S;m) = \{0, a_1, \ldots, a_{m-1}\}$ with $a_i \equiv i \bmod m$ for each $i$, and let $a_k = \max \Ap(S;m)$.  
\begin{enumerate}[(a)]
\item 
$S$ is symmetric if and only if $a_i \preceq a_k$ for each $i$.  

\item 
$S$ is pseudosymmetric if and only if there exists $j$ such that $2a_j = a_k + m$ and $a_i \preceq a_k$ for all $i \ne j$.  
\end{enumerate}
\end{lemma}

The following result is a direct consequence of \cite[Lemma~3.3]{r}, and characterizes when a given set is the Apéry set of a numerical semigroup.

\begin{lemma}\label{lem:is-Apery}
Let $m > 1$ be a integer, and let $A=\{0,a_1, \ldots, a_{m-1}\}$ with $a_i\equiv i \pmod m$. Then $S=\langle A\cup \{m\}\rangle$ is a numerical semigroup containing $m$.  Moreover, $\Ap(S;m) = A$ if and only if for every $i,j,k \in \{1, \ldots, m-1\}$ with $i + j \equiv k \bmod m$, the inequality $a_i+a_j \ge  a_k$ holds.
\end{lemma}

Fix numerical semigroups $S$ and $T$, and fix $m \in S \cap T$. Write 
\[
\Ap(S;m)=\{0,a_1,\dots,a_{m-1}\}
\quad \text{and} \quad
\Ap(T;m)=\{0,b_1,\dots,b_{m-1}\}
\]
with $a_i \equiv b_i \equiv i \bmod m$ for each $i$.  By the membership criterion~\eqref{eq:membership-with-apery}, we have
\begin{equation}\label{eq:inclusion-apery}
S \subseteq T \text{ if and only if } b_i \le a_i \text{ for each } i\in\{1,\dots,m-1\},
\end{equation}
and so 
\[
\Ap(S \cap T;m)= \{0, \max(a_1,b_1), \ldots, \max(a_{m-1},b_{m-1})\}.
\]

If $S\subseteq T$, then we define 
\[
\operatorname{M}_S(T)=\big\{i\in\{1,\dots,m-1\} : b_i = a_i\big\}.
\]
It follows easily that 
\begin{equation}\label{eq:char-M_S}
S=T_1\cap \dots \cap T_k \text{ if and only if } \{1,\dots,m-1\}=\bigcup_{i=1}^k \operatorname{M}_{S}(T_i).
\end{equation}
Thus, a decomposition $S=T_1\cap \dots \cap T_k$ is minimal if and only if $\{\operatorname{M}_S(T_i)\}_{i}$ is an irredundant cover of $\{1,\dots,m-1\}$.  

The set of gaps of $S$ that are maximal under $\preceq_S$ is known as the set of \emph{pseudo-Frobenius numbers} of $S$, and it is denoted by $\operatorname{PF}(S)$. 
We write $\operatorname{t}(S) = \#\operatorname{PF}(S)$, which is known as the \emph{type} of $S$.  It is well known that, for a numerical semigroup $S$ and nonzero $m \in S$,
\begin{equation}\label{eq:pseudo-Frobenius-Apery}
\operatorname{PF}(S)=-m+\operatorname{Maximals}_{\preceq_S}(\Ap(S;m)),
\end{equation}
so in particular $\operatorname{t}(S) \le m-1$ (see for instance \cite[Proposition~2.20]{numerical}).  

The set of \emph{special gaps} of $S$ (see \cite[Chapter~4]{numerical}) is defined as
\[
\operatorname{SG}(S)=\{x\in \operatorname{PF}(S) : 2x\in S\}.
\]
Equivalently, $x \in \operatorname{SG}(S)$ if and only if $S \cup \{x\}$ is a numerical semigroup.  
Note that since $\operatorname{SG}(S) \subseteq \operatorname{PF}(S)$, we have $\#\operatorname{SG}(S)\le \operatorname{t}(S) \le m-1$.

Special gaps are central to the study of decompositions of numerical semigroups.  If $S_1, \ldots, S_k \supseteq S$ are oversemigroups, then by \cite[Proposition~4.48]{numerical},
\begin{equation}\label{eq:char-decomp-SG}
	S=S_1\cap \dots \cap S_k \text{ if and only if }\operatorname{SG}(S)\subseteq \bigcup_{i=1}^k \operatorname{SG}(S_i). 	
\end{equation}
Moreover, the expression $S=S_1\cap \dots \cap S_k$ is irredundant if and only if for each $S_i$, we have $x \notin S_i$ for some $x \in \operatorname{SG}(S)$.  As such, any irreducible decomposition of $S$ has length at most $\#\operatorname{SG}(S)$.

\section{Sets of decomposition lengths in low multiplicity}\label{sec:low-multiplicity}

In this section we analyze the sets of decomposition lengths of numerical semigroups with multiplicity at most six, and show that they are always intervals.  

Fix a numerical semigroup $S$ and let $m = \operatorname{m}(S)$.  Any irredundant decomposition of $S$ has at most $\#\operatorname{SG}(S) \le m - 1$ components.  Thus, if $m \le 4$, then the only possible sets of decomposition lengths are $\{1\}$ (if $S$ is irreducible), $\{2\}$, $\{3\}$, or $\{2,3\}$ (if $S$ is not irreducible), all of which are intervals.  In fact, a bit more can be said, as discussed in the following remark.  

\begin{remark}\label{r:m4possibleintervals}
It turns out $\{3\}$ does not occur as the set of decomposition lengths for any numerical semigroup $S$ with multiplicity $4$, as if $S$ is not irreducible, then $S$ has a decomposition of length 2.  To~this end, it suffices to locate an irreducible $T \supseteq S$ with $\#M_T(S) = 2$, since only one additional component would then be needed to form a decomposition of $S$.  We claim that, with no assumptions on $S$ other than $\operatorname{m}(S) = 4$, there exists such a $T$ with $M_S(T) \supseteq \{1,3\}$.  Write $\Ap(S;4) = \{0,a_1,a_2,a_3\}$ with each $a_i \equiv i \bmod 4$.  
\begin{itemize}
\item 
If $a_1 \le a_3$, then let $b_2 = a_3 - a_1$ and take $T=\langle 4,a_1,b_2,a_3\rangle$.  Since $a_1 + a_2 \ge a_3$ by~\eqref{eq:membership-with-apery}, we get $a_2 \ge b_2$, and upon verifying 
\[
2a_1 \ge a_2 \ge b_2,
\quad
2a_3 \ge a_2 \ge b_2,
\quad
a_1 + b_2 = a_3,
\quad \text{and} \quad 
b_2 + a_3 \ge a_3 \ge a_1,
\]
we may apply Lemma~\ref{lem:is-Apery} to conclude $T$ is in fact a numerical semigroup with $\Ap(T)=\{0,a_1,b_2,a_3\}$.  Since $a_3 = b_1 + a_2$, we see $T$ is also symmetric by Lemma~\ref{lem:max-number-sg-bound-MST}, and by \eqref{eq:inclusion-apery}, we have $S \subseteq T$.

\item 
If $a_3 \le a_1$, then by an analogous application of Lemma~\ref{lem:is-Apery}, let $b_2 = a_1 - a_3$ and choose $T$ so that $\Ap(T;4) = \{0,a_1,b_2,a_3\}$.  

\end{itemize}
As such, the only possible sets of decomposition lengths for multiplicity 4 numerical semigroups are $\{1\}$, $\{2\}$, and $\{2,3\}$.  
\end{remark}




For $m = 5$, a bit more work is needed. 
We first give a lemma that yields a bound on $\#\operatorname{M}_S(T)$ when $S$ has maximal number of special gaps.

\begin{lemma}\label{lem:max-number-sg-bound-MST}
Suppose $m = \operatorname{m}(S)$, $\#\operatorname{SG}(S)=m-1$, and let $T$ be irreducible with $S \subseteq T$.  Then $\#\operatorname{M}_S(T) \le \tfrac{1}{2}m$.
\end{lemma}

\begin{proof}
Write $\Ap(S) = \{0, a_1, \ldots, a_{m-1}\}$ and $\Ap(T) = \{0, b_1, \ldots, b_{m-1}\}$ with each $a_i\equiv b_i \equiv i \bmod m$, and let $k \in \{1, \ldots, m-1\}$ so that $\operatorname{F}(T) = b_k - m$.  Notice that we cannot have $b_i + b_j = b_k$ with $i, j \in \operatorname{M}_S(T)$, as 
\[a_k \le a_i + a_j = b_i + b_j = b_k \le a_k\]
would force $a_i + a_j = a_k$, contradicting the assumption that $a_i - m \in \operatorname{SG}(S)$ since $a_j \in S$ and $(a_i - m) + a_j = a_k - m \notin S$ (and so $a_i-m\not\in \operatorname{PF}(S)$).

Now, consider the partition of $\{0, \ldots, m-1\}$ whose blocks have the form $\{i,j\}$ for (not necessarily distinct) $i, j$ with $i + j \equiv k \bmod m$.  The number of singleton blocks in this partition depends on the number of solutions $x \in \{0,\dots,m-1\}$ of the equation $2x \equiv k \bmod m$. If $m$ is odd, then this equation as a single solution, and thus the partition has $\tfrac{1}{2}(m+1)$ blocks. If $m$ is even, then the equation has either zero or two solutions. In the first case, the partition has $\tfrac{1}{2}m$ blocks, and in the second case, it has $\tfrac{1}{2}(m+2)$ blocks.	

Observe that each block considered in the preceding paragraph can contain at most one element of $\operatorname{M}_S(T)$, since if $\{i,j\} \subseteq \operatorname{M}_S(T)$ with $i \ne j$ then it would force $b_i + b_j = b_k$, which we showed in the first paragraph cannot hold.

Towards a contradiction, assume $\#\operatorname{M}_S(T) > \tfrac{1}{2}m$. 
\begin{itemize}
	\item If $m$ is odd, this forces $\#\operatorname{M}_S(T) = \tfrac{1}{2}(m+1)$, with each block of the partition containing exactly one element of $\operatorname{M}_S(T)$, and only one block being a singleton. Let $\{i\}$ be this singleton block. Then, $2b_i = b_k$ or $2b_i = b_k + m$. We have seen that the first case is impossible. The block $\{0,k\}$ must contain an element of $\operatorname{M}_S(T)$, which must be $k$, meaning that $b_k = a_k$. Thus, in the second case, $a_k\le 2a_i\le 2b_i = b_k=a_k$. This yields $a_k = 2a_i$, and then $a_i+(a_i-m)=a_k - m$, which implies that $a_i - m \notin \operatorname{SG}(S)$, a contradiction.
     
	\item If $m$ is even, then we must have $\#\operatorname{M}_S(T) = \tfrac{1}{2}(m+2)$ and two distinct singleton blocks $\{i\}$ and $\{j\}$ with $i, j \in \operatorname{M}_S(T)$; this is only possible if $T$ is pseudosymmetric with $\{2b_i, 2b_j\} = \{b_k, b_k + m\}$, meaning $2b_i = b_k$ or $2b_j = b_k$, either of which is a contradiction. 
\end{itemize}
In each case, we arrive at a contradiction, so we conclude $\#\operatorname{M}_S(T) \le \tfrac{1}{2}m$.  
\end{proof}

\begin{thm}\label{t:m5interval}
	Let $S$ be a numerical semigroup with multiplicity $m = 5$. Then, the set of decomposition lengths of $S$ into irreducibles is an interval.
\end{thm}

\begin{proof}

	If $S$ is irreducible, then its set of decomposition lengths is $\{1\}$.  Otherwise, it is contained in $\{2,3,4\}$, so it suffices to prove the set of decomposition lengths of $S$ does not equal $\{2,4\}$.  To this end, suppose $S = S_1 \cap \cdots \cap S_4 = T_1 \cap T_2$ are decompositions of $S$.  Since the first decomposition is irredundant, by \eqref{eq:char-decomp-SG}, $\#\operatorname{SG}(S) = 4$.  By~Lemma~\ref{lem:max-number-sg-bound-MST}, $\#\operatorname{M}_S(T_1) \le 2$ and $\#\operatorname{M}_S(T_2) \le 2$, but since $\#(\operatorname{M}_S(T_1) \cup \operatorname{M}_S(T_2)) = 4$, we must have $\#\operatorname{M}_S(T_1) = \#\operatorname{M}_S(T_2) = 2$.  Since the sets $\operatorname{M}_S(S_i)$ form a partition of $\{1,2,3,4\}$ into singletons, there exist distinct $i_1, i_2 \in \{1,2,3,4\}$ so $S=T_1\cap S_{i_1}\cap S_{i_2}$ is a decomposition of $S$ with three factors by \eqref{eq:char-decomp-SG}.  
\end{proof}

Next, we turn our attention to numerical semigroups $S$ with $\operatorname{m}(S) = 6$.  We first prove the following technical lemma that, in similar fashion to Remark~\ref{r:m4possibleintervals}, ensures the existence of certain irreducible oversemigroups $T$ for which the set $\operatorname{M}_S(T)$ is tightly controlled.  

\begin{lemma}\label{l:m6interval}
Fix a numerical semigroup $S$ with $\Ap(S) = \{0, a_1, \ldots, a_5\}$.  
\begin{enumerate}[(a)]
\item 
There exist irreducible semigroups $T, T' \supseteq S$ with $2, 5 \in \operatorname{M}_S(T)$ and $1, 4 \in \operatorname{M}_S(T')$.  

\item 
If $\#\operatorname{SG}(S) \ge 4$, then $T$ and $T'$ can be chosen so that $\#\operatorname{M}_S(T) \le 3$, $\#\operatorname{M}_S(T') \le 3$, and $\operatorname{M}_S(T) \textcolor{blue}{\cup} \operatorname{M}_S(T') \subseteq \{1,2,4,5\}$.  



\item If $\#\operatorname{SG}(S) = 5$, then $T$ and $T'$ can be chosen so that additionally $\#\operatorname{M}_S(T) = 2$ or $\#\operatorname{M}_S(T') = 2$.
\end{enumerate}
\end{lemma}

\begin{proof}
We begin by constructing $T$.  First suppose $a_2 > a_5$.  Define $b_1, b_3, b_4 \in \mathbb{Z}_{\ge 1}$ so that each $b_i \equiv i \bmod 6$,
$$
b_3 = a_2 - a_5,
\qquad \text{and} \qquad
\{2b_1, 2b_4\} = \{a_2, a_2 + 6\},
$$
and let $T = \<6, b_1, a_2, b_3, b_4, a_5\>$.  
We use Lemma~\ref{lem:is-Apery} to show that 
\[
\Ap(T) = \{0, b_1, a_2, b_3, b_4, a_5\}.
\]
Omitting the cases where $i+j=6\equiv 0 \pmod 6$), we see
\begin{itemize}
\item $b_1 +b_1 \ge a_2$, since $2b_1 \in \{a_2, a_2 + 6\}$;
\item $b_1+a_2 \ge b_3$, since $a_2=b_3 + a_5\ge b_3$;
\item $b_1+b_3 \ge b_1 + 6 \ge b_4$, since $|b_1-b_4|=3$;
\item $b_1+b_4 = \frac{1}{2}(2a_2 + 6) = a_2 + 3 \ge a_5$ (by the hypothesis $a_2 > a_5$);
\item $a_2+a_2 \ge a_2 + 6 \ge b_4$:
\item $a_2+b_3 \ge a_2 \ge a_5$;
\item $a_2+a_5 \ge a_2+6 \ge b_1$;
\item $b_3+b_4 \ge 6+b_4\ge b_1$, since $|b_1-b_4|=3$;
\item $b_3+a_5 = a_2\ge a_2$;
\item $b_4+b_4 \ge a_2$, since $2b_4 \in \{a_2, a_2 + 6\}$;
\item $b_4+a_5 \ge b_3$, since $b_4\ge \frac{1}{2}a_2$ (by definition) and $4a_5 \ge a_2$ (by \eqref{eq:membership-with-apery}) together imply $b_4+a_5\ge \frac{1}{2}a_2 + \frac{1}{4}a_2 = a_2 - \tfrac{1}{4}a_2 \ge a_2 - a_5 = b_3$; and
\item $a_5+a_5\ge b_4$, since $2a_5\ge a_4$ and $2b_4\le a_2+6\le 2a_4+6=2(a_4+3)$ (by~\eqref{eq:membership-with-apery}) imply $b_4 \le a_4 + 3$, and $b_4 \equiv a_4 \equiv 4 \bmod 6$ then yields $b_4 \le a_4$.
\end{itemize}
With these inequalities, we conclude that $\Ap(T) = \{0, b_1, a_2, b_3, b_4, a_5\}$. 

Now we prove that $S\subseteq T$, and to this end we use \eqref{eq:inclusion-apery} by showing that $b_1\le a_1$, $b_3 \le a_3$, and $b_4 \le a_4$.  We have already proved that $b_4\le a_4$, so let us focus on the other two inequalities. Since $2b_1\le a_2+6\le 2a_1+6=2(a_1+3)$ (we are using that $a_2 \le 2a_1$ by \eqref{eq:membership-with-apery}), we have $b_1 \le a_1 + 3$, and both are congruent to $1$ mod $6$, so $b_1 \le a_1$.
Also, we have that $a_3+a_5\ge a_2$ (by \eqref{eq:membership-with-apery}), and so $a_3\ge a_2-a_5=b_3$. 

Next, we show that $T$ is pseudo-symmetric.  Notice that $b_3+a_5=a_2$ and so $b_3 \preceq_S a_2$ and $a_5 \preceq_S a_2$. We distinguish two cases.  If $2b_1 = a_2$, then $2b_4 = a_2 + 6$, and $b_1\preceq_S a_2$. By Lemma~\ref{lem:pseudosymmposet}, $T$ is pseudo-symmetric. If $2b_4 = a_2$, then $2b_1 = a_2 + 6$, and we argue as in the previous case.

We now verify the claims in part~(b) regarding $T$ in this case.  
If~$b_3 = a_3$, then $a_2 = a_3 + a_5$.  This means $\operatorname{Ap}(S)$ has at most three maximal elements under $\preceq_S$, implying $\#\operatorname{SG}(S) \le \#\operatorname{PF}(S)\le 3$.  Also, if $b_i = a_i$ for $i \in \{1,4\}$, then $2(a_i - 6) < a_2$ and thus $2(a_i-6)\notin S$ by~\eqref{eq:membership-with-apery}, so $a_i - 6 \notin \operatorname{SG}(S)$.  We conclude if $\#\operatorname{SG}(S) \ge 4$, then either $a_1 > b_1$ or $a_4 > b_4$, so $\#\operatorname{M}_S(T) \le 3$ and $3 \notin \operatorname{M}_S(T)$.  
We also note that 
\begin{equation}\label{eq:m6lemmacase1extra}
\text{in this case}, \quad \text{if } \#\operatorname{SG}(S) = 5, \quad \text{then } M_S(T) = \{2,5\},
\end{equation}
which will be important in the proof of part~(c).  

Next, suppose $a_5 > a_2$.  Let $b_3 = a_5 - a_2$, and fix $b_1, b_4 \ge 0$ so each $b_i \equiv i \bmod 6$,
$$
a_1 \ge b_1 \ge \tfrac{1}{2}a_2,
\qquad
a_4 \ge b_4 \ge \tfrac{1}{2}a_2,
\qquad \text{and} \qquad
b_1 + b_4 = a_5;
$$
note this is always possible since $a_1 + a_4 \ge a_5$ and $\tfrac{1}{2}a_2 + \tfrac{1}{2}a_2 = a_2 \le a_5$.  
As in the previous case, we claim $T = \<6, b_1, a_2, b_3, b_4, a_5\>$ is an oversemigroup of $S$ with $\Ap(T) = \{0, b_1, a_2, b_3, b_4, a_5\}$.  

We start by proving that $\Ap(T) = \{0, b_1, a_2, b_3, b_4, a_5\}$, with the help of Lemma~\ref{lem:is-Apery}:
\begin{itemize}
	\item $b_1 +b_1\ge a_2$, since $b_1 \ge \tfrac{1}{2}a_2$;
	\item $b_1 + a_2 \ge b_1 + (b_4 - a_2) = a_5 - a_2 = b_3$, since $2a_2\ge a_4$ by \eqref{eq:membership-with-apery} and $a_4 \ge b_4$;
	\item $b_1 + b_3 = b_1 + (a_5 - a_2) \ge a_5 - b_1 = b_4$, since $b_3= a_5 - a_2$ and $b_1\ge \frac{1}{2}a_2$;
	\item $b_1+b_4 = a_5\ge a_5$;
	\item $a_2 +a_2 \ge b_4$, since $b_4 \le a_4$ and $2a_2 \ge a_4$ by \eqref{eq:membership-with-apery};
	\item $a_2 + b_3 =a_2 + (a_5 - a_2) = a_5 \ge a_5$;
	\item $a_2 + a_5 \ge a_1 \ge b_1$ by \eqref{eq:membership-with-apery};
	\item $b_3 + b_4 = (a_5 - a_2) + b_4 \ge a_5 - b_4 = b_1$, since $b_4 \ge \frac{1}{2}a_2$;
	\item $b_3 + a_5 \ge a_5\ge a_2$;
	\item $b_4+b_4 \ge a_2$, since $b_4 \ge \tfrac{1}{2}a_2$;
	\item $b_4 + a_5 = b_4 + (a_2 + b_3) \ge b_3$; and
	\item $a_5+a_5 \ge a_4 \ge b_4$ by \eqref{eq:membership-with-apery}.
\end{itemize}

Now, since $b_1\le a_1$ and $b_4 \le a_4$, in order to prove $S\subseteq T$ using \eqref{eq:inclusion-apery}, it suffices to show that $b_3 \le a_3$.  This follows from the fact that $a_3 + a_2 \ge a_5$ by \eqref{eq:membership-with-apery}, whence $a_3 \ge a_5 - a_2 = b_3$.

We now prove that $T$ is symmetric.  Notice that $b_3 + a_2 = a_5$, so $b_3 \preceq_S a_2$ and $a_5 \preceq_S a_2$ hold. Also, $b_1+b_4 = a_5$, and thus $b_1 \preceq_S a_5$ and $b_4 \preceq_S a_5$.  Thus, by Lemma~\ref{lem:pseudosymmposet}, $T$ is symmetric.  This proves part~(a).

We now verify the claims in part~(b) regarding $T$ in this case.  If $a_1 + a_4 = a_5$, then $\#\operatorname{SG}(S)\le \#\operatorname{PF}(S)\le 3$.  Thus, if $\#\operatorname{SG}(S) \ge 4$, then $a_1 + a_4\ge a_5 + 6$ by~\eqref{eq:membership-with-apery}.
Also, notice that if both $a_1 - 6 \notin \operatorname{SG}(S)$ and $a_4 - 6 \notin \operatorname{SG}(S)$, then $\#\operatorname{SG}(S) \le 3$. As~we are assuming $\#\operatorname{SG}(S) \ge 4$, at least one of $a_1 - 6$ or $a_4 - 6$ must lie in $\operatorname{SG}(S)$.  If $a_i-6\in \operatorname{SG}(S)$  for $i\in \{1,4\}$, then $2(a_i - 6) \ge a_2$ by~\eqref{eq:membership-with-apery} since $2(a_i-6)\in S$, and so $a_i - 6 \ge \frac{1}{2}a_2$. We now consider all possible cases.  

\begin{itemize}
\item 
Suppse $a_1-6\in \operatorname{SG}(S)$ and $a_4 - 6 \notin \operatorname{SG}(S)$.  Choosing $b_1=a_1-6$ and then $b_4=a_5-b_1=a_5-a_1+6$, we get $b_4=a_5-a_1+6\le a_4$ since $a_1+a_4\ge a_5+6$.  In this setting, $\{2,5\} \subseteq \operatorname{M}_S(T) \subseteq \{2,5,4\}$, and we have $\{2,5,4\}=\operatorname{M}_S(T)$ if and only if $a_4=b_4$, which is equivalent to $a_1+a_4=a_5+6$.  

\item 
Suppose $a_4-6\in \operatorname{SG}(S)$ and $a_1 - 6 \notin \operatorname{SG}(S)$. Choosing $b_4=a_4-6$ and then $b_1=a_5-b_4=a_5-a_4+6$, we get $b_1=a_5-a_4+6\le a_1$ since $a_1+a_4\ge a_5+6$.  In this case, $\{2,5\} \subseteq \operatorname{M}_S(T)\subseteq \{1,2,5\}$, and we have $\{1,2,5\}=\operatorname{M}_S(T)$ if and only if $a_1=b_1$, which is equivalent to $a_1+a_4=a_5+6$.

\item Finally, suppose $a_1 - 6, a_4 - 6 \in \operatorname{SG}(S)$. 
If $a_1 + a_4 \ge a_5 + 12$, then we can choose $b_i = a_i - 6$ for both $i = 1$ and $i = 4$, yielding $\operatorname{M}_S(T) = \{2,5\}$.
If, on the other hand,  $a_1 + a_4 = a_5 + 6$, we proceed as in one of the first two cases:\ choose $b_1=a_1-6$ and $b_4=a_5-b_1=a_5-a_1+6=a_4$, obtaining $\operatorname{M}_S(T) = \{2,5,4\}$; or choose $b_4=a_4-6$ and $b_1=a_5-b_4=a_5-a_4+6=a_1$, obtaining $\operatorname{M}_S(T) = \{1,2,5\}$. 

\end{itemize}
In each case, we conclude if $\#\operatorname{SG}(S) \ge 4$, then $\#\operatorname{M}_S(T) \le 3$ and $3 \notin \operatorname{M}_S(T)$.   Also,
\begin{equation}\label{eq:m6lemmacase2extra}
\text{in this case}, \quad \text{if } \#\operatorname{SG}(S) = 5, \quad \text{then } M_S(T) = \{2,5\} \text{ or } a_1 + a_4 = a_5 + 6,
\end{equation}
which will be important in the proof of part~(c).  

This completes the construction of $T$.  To construct $T'$, simply apply the nontrivial automorphism of $\ZZ_6$ to all subscripts in the proof thus far (under this automorphism, the condition $a_2>a_5$ becomes $a_4>a_1$ and $a_1+a_4=a_5+6$ becomes $a_2+a_5=a_1+6$).

Having now proven parts~(a) and~(b), and the only remaining claim to prove is that if $\#\operatorname{SG}(S) = 5$, then $\#\operatorname{M}_S(T) = 2$ or $\#\operatorname{M}_S(T') = 2$.  
If $T$ cannot be chosen so that $\#\operatorname{M}_S(T) = 2$, then by~\eqref{eq:m6lemmacase1extra}, we must have $a_5 > a_2$, and thus by~\eqref{eq:m6lemmacase2extra}, we must have $a_1 + a_4 = a_5 + 6$.  In particular $a_5 > a_1$.  As such, $a_2 + a_5 > a_1 + 6$, meaning $T'$ can be chosen so that $\#\operatorname{M}_S(T') = 2$ by~\eqref{eq:m6lemmacase1extra} and~\eqref{eq:m6lemmacase2extra} from the construction of $T'$.  
\end{proof}

\begin{thm}\label{t:m6interval}
Let $S$ be a numerical semigroup with multiplicity six.  
\begin{enumerate}[(a)]
\item 
If $S$ has a decomposition of length 5, then $S$ has a decomposition of length 4.  
\item 
If $S$ has a decomposition of length 4, then $S$ has a decomposition of length 3.  
\end{enumerate}
In particular, the set of decomposition lengths of $S$ is an interval.
\end{thm}

\begin{proof}
Let $T$ and $T'$ denote the oversemigroups of $S$ constructed in Lemma~\ref{l:m6interval}(a).  

First, suppose $S = T_1 \cap \cdots \cap T_5$ is an irreducible decomposition. 
After rearranging, by \eqref{eq:char-M_S}, we may assume $\operatorname{M}_S(T_i) = \{i\}$ for each $i$. Since this decomposition is irredundant, $\#\operatorname{SG}(S) = 5$ (as a consequence of \eqref{eq:char-decomp-SG}).  Thus, by Lemma~\ref{l:m6interval}(c), either $\operatorname{M}_S(T) = \{2,5\}$, in which case $S = T \cap T_1\cap T_3 \cap T_4$ is a decomposition, or $\operatorname{M}_S(T') = \{1,4\}$, in which case $S = T' \cap T_2\cap T_3 \cap T_5$ is a decomposition.  This proves part~(a).  

Next, suppose $S$ has a decomposition of length four, and thus also $\#\operatorname{SG}(S) \ge 4$. By Lemma~\ref{l:m6interval}(a) and~(b), $T$ and $T'$ can be chosen so that $\operatorname{M}_S(T)\cup\operatorname{M}_S(T') = \{1,2,4,5\}$. 
As $S$ admits a decomposition of length four, \eqref{eq:char-M_S} implies there exists an irreducible oversemigroup $U$ of $S$ such that $3 \in \operatorname{M}_S(U)$ and $\#\operatorname{M}_S(U) \le 2$.  We thus obtain a length-3 decomposition $S = T\cap T'\cap U$, wherein neither $T$ nor $T'$ is redundant by Lemma~\ref{l:m6interval}(b).  This proves part~(b).  

The final claim follows from examining the non-interval subsets of $\{2,3,4,5\}$.
\end{proof}

In the remainder of this section, we give some computational results regarding which sets of decomposition lenghts occur in small multiplicity.  

\begin{example}\label{e:m5possibleintervals}
Every interval subset of $\{2,3,4\}$ occurs as the set of decomposition lengths of a multiplicity 5 numerical semigroup.  For instance, the semigroups
\begin{align*}
&S_{2} = \<5, 6, 7\>,
&
&S_{3} = \<5, 6, 13, 14\>,
&
&S_{4} = \<5, 11, 12, 13, 14\>,
\\
&S_{23} = \<5, 11, 13, 19\>,
&
&S_{34} = \<5, 12, 13, 14, 16\>,
&
&S_{234} = \<5, 21, 22, 33, 34\>
\end{align*}
achieve each such interval of decomposition lengths.  
\end{example}

\begin{example}\label{e:m6possibleintervals}
Theorem~\ref{t:m6interval} identifies the only restrictions on the set of decomposition lengths of a multiplicity-6 numerical semigroup.  Indeed, the semigroups
\begin{align*}
&S_{2} = \<6, 7, 10\>,
&
&S_{34} = \<6, 7, 15, 16, 17\>,
&
&S_{345} = \<6, 13, 14, 15, 16, 17\>,
\\
&S_{3} = \<6, 7, 8, 17\>,
&
&S_{234} = \<6, 8, 13, 15, 17\>,
&
&S_{2345} = \<6, 16, 14, 19, 21, 23\>,
\\
&S_{23} = \<6, 7, 9, 17\>
\end{align*}
achieve every interval subset of $\{2,3,4,5\}$ as a set of decomposition lengths, except those sets forbidden by Theorem~\ref{t:m6interval}.  
\end{example}

\begin{example}\label{e:m7possibleintervals}
Every interval subset of $\{2,3,4,5,6\}$ arises as the set of decomposition lengths of some numerical semigroup with multiplicity $m = 7$.  For instance,
\begin{align*}
&S_{2} = \<7, 8, 9\>,
&
&S_{234} = \<7, 15, 17, 18, 26\>,
&
&S_{6} = \<7, 22, 23, 24, 25, 26, 27\>,
\\
&S_{3} = \<7, 8, 9, 10\>,
&
&S_{5} = \<7, 15, 26, 27, 31, 32\>,
&
&S_{56} = \<7, 16, 17, 18, 19, 20, 22\>,
\\
&S_{23} = \<7, 8, 10, 11\>,
&
&S_{45} = \<7, 8, 17, 18, 19, 20\>,
&
&S_{456} = \<7, 15, 16, 17, 18, 19, 20\>,
\\
&S_{4} = \<7, 15, 17, 33\>,
&
&S_{345} = \<7, 10, 15, 16, 18, 19\>,
&
&S_{3456} = \<7, 16, 18, 20, 22, 24, 26\>,
\\
&S_{34} = \<7, 8, 10, 19\>,
&
&S_{2345} = \<7, 15, 18, 24, 26, 34\>,
&
&S_{23456} = \<7, 24, 25, 27, 30, 36, 40\>
\end{align*}
achieve each such interval of decomposition lengths.  
\end{example}

In addition to aligning with Remark~\ref{r:m4possibleintervals} and Theorem~\ref{t:m6interval}, computational evidence suggests the following conjecture is true for $m = 8$ as well.  In fact, via a similar search to the ones that yielded Examples~\ref{e:m5possibleintervals}-\ref{e:m7possibleintervals}, if Conjecture~\ref{conj:m8possibleintervals} is true, then it comprises the only restrictions on decomposition lengths in the $m = 8$ case.  

\begin{conj}\label{conj:m8possibleintervals}
Fix a numerical semigroup $S$ of multiplicity $m = 2k$ with $k \in \ZZ$.  If $S$ has a decomposition of length $j$ with $k < j < m$, then $S$ has a decomposition of length $j-1$.  



\end{conj}

\section{Ordinary numerical semigroups}
\label{sec:ordinary}

Recall that the \emph{ordinary} numerical semigroup with multiplicity $m$ is 
$$
H_m = \<m, m + 1, \ldots, 2m - 1\> = \ZZ_{\ge 0} \setminus \{1, 2, \ldots, m - 1\}
$$
and has $\operatorname{F}(H_m) = m-1$, $\operatorname{SG}(H_m) = [\tfrac{1}{2}m, m-1] \cap \ZZ$, and $|\operatorname{SG}(H_m)| = \lfloor \tfrac{1}{2}m \rfloor$.  In~this section, we construct irreducible decompositions of $H_m$ demonstrating that the number of distinct lengths asymptotically approaches $|\operatorname{SG}(H_m)|$.

We first describe the irreducible components we will be using.  Define
$$T(F) = \ZZ_{\ge 0} \setminus ([1, \tfrac{1}{2}F] \cup \{F\}),$$
and given $j, F \in \ZZ_{\ge 0}$ with $F = 2^j (2k+1)$ for $k \in \ZZ_{\ge 0}$, define
$$I(F) = (\<2^{j+1}\> + T(F)) \setminus \{2^j (2k'+1) : 0 \le k' \le k\}.$$

Table~\ref{tb:ordinarycomponents} contains several examples of $T(F)$ and $I(F)$.  Intuitively, $I(F)$ is obtained from $T(F)$ by removing some elements $n \in (\tfrac{1}{2}F, F) \subseteq T(F)$ and subsequently adding each $F - n$, so as to preserve irreducibility.  Lemma~\ref{l:ordinarydecomp} below ensures the resulting set remains closed under addition.  
Note in particular that $I(F) = \<2, F+2\>$ for any odd $F$.

\begin{table}[t]
\begin{center}
\begin{tabular}{l|*{24}{r@{\phantom{\,\,\,}}}r}
\hline
$T(20)$
&0& & & & & & & & & &  &11&12&13&14&15&16&17&18&19&  &21&$\rightarrow$  
\\
$I(20)$
&0& & & & & & & &8& &  &11&  &13&14&15&16&17&18&19&  &21&$\rightarrow$  
\\[0.4em]
$T(21)$
&0& & & & & & & & & &  &11&12&13&14&15&16&17&18&19&20&  &22&$\rightarrow$  
\\
$I(21)$
&0& &2& &4& &6& &8& &10&  &12&  &14&  &16&  &18&  &20&  &22&$\rightarrow$  
\\[0.4em]
$T(22)$
&0& & & & & & & & & &  &  &12&13&14&15&16&17&18&19&20&21&  &23&$\rightarrow$
\\
$I(22)$
&0& & & &4& & & &8& &  &  &12&13&  &15&16&17&  &19&20&21&  &23&$\rightarrow$
\end{tabular}
\end{center}
\caption{Elements of several semigroups of the form $T(F)$ and $I(F)$.}
\label{tb:ordinarycomponents}
\end{table}

\begin{lemma}\label{l:ordinarydecomp}
Given $j$ and $F$ as above, the following hold:
\begin{enumerate}[(a)]
\item 
$I(F)$ is a numerical semigroup with $\operatorname F(I(F)) = F$;

\item 
$I(F)$ is symmetric if $j = 0$ and pseudosymmetric otherwise; and

\item 
each $a \in (\tfrac{1}{2}F, F]$ lies in $I(F)$ if and only if $a \not\equiv 2^j \bmod 2^{j+1}$.  

\end{enumerate}
\end{lemma}

\begin{proof}
Let $A = \{2^j \cdot (2k'+1) : 0 \le k' \le k\} = \{2^j, 2^j \cdot 3, \ldots, F - 2^{j+1}, F\}$.  We first briefly prove~(c):\ notice that $T(F) \subseteq \<2^{i+1}\> + T(F)$, so any gaps of $I(F)$ in $(\tfrac{1}{2}F, F]$ must lie in $A$.  

Now, it is known that $T(F)$ is a numerical semigroup, so clearly $\<2^{j+1}\> + T(F)$ is a numerical semigroup as well.  To prove $I(F)$ is closed under addition, it suffices to fix $a, b \in I(F)$ and prove $a + b \notin A$.  To this end, consider the following cases:
\begin{itemize}
\item 
if $a, b > F/2$, then $a + b > F = \max A$;

\item 
if $a < F/2$ and $b > F/2$, then $2^{j+1} \mid a$ and $b \not\equiv 2^j \bmod 2^{j+1}$, so we have $a + b \equiv b \bmod 2^{j+1}$ and thus $a + b \notin A$; and

\item 
if $a, b < F/2$, then $2^{j+1} \mid a, b$, so $2^{j+1} \mid (a + b)$ and thus $a + b \notin A$.  

\end{itemize}
This proves $I(F)$ is a numerical semigroup, and since $\operatorname F(T(F)) = F = \max A$, we have $\operatorname F(I(F)) = F$, thus verifying~(a).  

This leaves~(b).  Notice that $I(F)=(T(F)\setminus A)\cup (F-A)$ and $\operatorname{F}(I(F))=\operatorname{F}(T(F))$. Thus, $I(F)$ and $T(F)$ share genus and Frobenius number.  Since $T(F)$ is irreducible, so is $I(F)$ by \cite[Corollary~4.5]{numerical}.
\end{proof}

\begin{example}\label{e:ordinaryminlen}
The proof of Theorem~\ref{t:ordinarydecomp} yields, among others, the following decompositions of $H_{28}$:
\begin{align*}
H_{28}
&= I(16) \cap I(24) \cap I(20) \cap I(26) \cap I(27)
\\
&= I(16) \cap I(24) \cap I(20) \cap I(26) \cap I(25) \cap T(27)
\\
&= I(16) \cap I(24) \cap I(20) \cap I(22) \cap I(25) \cap T(27) \cap T(26)
\\
&= I(16) \cap I(24) \cap I(20) \cap I(22) \cap I(23) \cap T(27) \cap T(26) \cap T(25)
\\
&= I(16) \phantom{{}\cap{}} \phantom{I(24)} \cap I(20) \cap I(22) \cap I(23) \cap T(27) \cap T(26) \cap T(25) \cap T(24).
\end{align*}
The first decomposition above amounts to partitioning
$$
\operatorname{SG}(H_{28}) = \{16\} \cup \{24\} \cup \{20\} \cup \{14,18,22,26\} \cup \{15,17,\ldots,27\}
$$
based on the 2-adic valuation of each element, and each subsequent decomposition splits the last element off of one block.  The only exception is the last line, wherein $I(24)$ is redundant after including $T(24)$ since both avoid only $24 \in \operatorname{SG}(H_{28})$.  
\end{example}

\begin{thm}\label{t:ordinarydecomp}
The ordinary numerical semigroup $H_m$ has a decomposition into irreducibles of each length in $[n_m, \tfrac{1}{2}m]$, where $n_m = \lfloor \log_2(\tfrac{1}{3}(m-1)) \rfloor + 2$.  
\end{thm}


\begin{proof}
We begin by exhibiting a decomposition of the smallest claimed length.  Let 
\[J = \{i \in \ZZ_{\ge 0} : 2^i (2k + 1) \in \operatorname{SG}(H_m) \text{ for some } k \in \ZZ_{\ge 0}\}.\]
We claim $|J| = n_m$.  Consider the following cases.  
\begin{itemize}
\item 
If $2^j \le \operatorname{F}(H_m) < 2^j + 2^{j-1}$ for some $j \in \ZZ$, then we claim $J = \{0, 1, \ldots, j - 2, j\}$.  Indeed, $2^j \in \operatorname{SG}(H_m)$ implies $j \in J$, 
$$2^{j-1} < \min \operatorname{SG}(H_m) \le \max \operatorname{SG}(H_m) < 3 \cdot 2^{j-1}$$
implies $j - 1 \notin J$, and for each $i \le j - 2$, the positive integers of the form $2^i(2k + 1)$ with $k \in \ZZ$ form an arithmetic sequence of step size $2^{i+1}$, so since  $\operatorname{SG}(H_m)$ is an interval of length $\tfrac{1}{2}m \ge 2^{j-1}$, we conclude $i \in J$.  
Having now shown $J$ has the claimed form, since $2^j \le m - 1 \le 2^j + 2^{j-1}$, we see 
\[
2^{j-2} \le \tfrac{1}{3}(m-1) < 2^{j-1},
\]
so $n_m = (j-2) + 2 = j = |J|$.  

\item 
In all remaining cases, we must have $2^j + 2^{j-1} \le \operatorname{F}(H_m) < 2^{j+1}$ for some $j \in \ZZ_{\ge 0}$, and by an analogous argument, $J = \{0, 1, \ldots, j\}$ and $|J| = n_m$.  
\end{itemize}

Next, for each $j \in J$, let 
\[F_j = \max(\{g \in \operatorname{SG}(H_m) : g = 2^j (2k + 1) \text{ for some } k \in \ZZ_{\ge 0}\}).\]
We claim 
\[H_m = \bigcap_{j \in J} I(F_j),\]
each component of which is irredudible by Lemma~\ref{l:ordinarydecomp}(b).  
Indeed, for each $j$, 
$$\operatorname F(I(F_j)) = F_j \le m - 1 = \operatorname F(H_m),$$
so each component contains $H_m$.  Each special gap $g \in \operatorname{SG}(H_m)$ can be written as $g = 2^j (2k + 1)$ for some $k \in \ZZ_{\ge 0}$ and $j \in J$ by the definition of $J$; since $g \ge \tfrac{1}{2}m \ge \tfrac{1}{2}F_j$ and $g \equiv 2^j \bmod 2^{j+1}$, we have $g \notin I(F_j)$ by Lemma~\ref{l:ordinarydecomp}(c).  This proves the claimed equality by \eqref{eq:char-decomp-SG}.
Note that each special gap of $H_m$ is a gap of exactly one component of the above constructed decomposition: if $g=2^j(2t+1)$ for some positive integer $t$, then $g\not\in I(F_j)$ and it will belong to the rest of components. 

Write $\operatorname{SG}(H_m) = \{g_1 > g_2 > \cdots\}$.  For each $\ell \le \tfrac{1}{2}(m-1) = |\operatorname{SG}(H_m)|$, denote by $D_\ell$ the decomposition
$$H_m = \bigg( \bigcap_{j \in J'} I(F_j') \bigg) \cap \bigg( \bigcap_{i = 1}^\ell T(g_i) \bigg),$$
where
$$J' = \{j : 2^j (2k + 1) \in \{g_{\ell+1}, g_{\ell+2}, \ldots\} \text{ for some } k \in \ZZ_{\ge 0}\}$$
and for each $j \in J'$,  
$$F_j' = \max(\{g \in \{g_{\ell+1}, g_{\ell+2}, \ldots\} : g = 2^j (2k + 1) \text{ for some } k \in \ZZ_{\ge 0}\}.$$
To prove this is a decomposition, we compare $D_\ell$ to $D_{\ell+1}$, where $D_0$ denotes the decomposition constructed in the first part of the proof.  Notice that $D_\ell$ and $D_{\ell+1}$ only differ in two ways:\ (i) the latter has an extra component $T(g_{\ell+1})$, and (ii) the component $I(F_j')$ in $D_\ell$ with $F_j' = g_{\ell+1}$ is either replaced by $I(F_j' - 2^{i+1})$ in $D_{\ell+1}$  or omitted, based on whether $F_j' - 2^{i+1} \in \operatorname{SG}(H_m)$.  In particular, every element of $\operatorname{SG}(H_m)$ is covered by an identical component of $D_\ell$ and $D_{\ell+1}$ except $g_{\ell+1}$, and it is still covered in $D_{\ell+1}$.  By induction, we conclude each $D_\ell$ is a decomposition of $H_m$.  

Lastly, we examine lengths of the $D_\ell$.  For each $\ell$, the decomposition $D_{\ell+1}$ either has the same number of components of $D_\ell$ or exactly one more component.  Moreover, the final decomposition is 
$$H_m = \!\!\!\!\!\!\bigcap_{g \in \operatorname{SG}(H_m)} \!\!\!\!\!\! T(g),$$
which has $|\operatorname{SG}(H_m)| =  \lfloor \tfrac{1}{2}m \rfloor$ components.  This completes the proof.  
\end{proof}

\begin{remark}\label{r:ordinarynotactualmin}
The decomposition
$$
H_{28} = I(27) \cap I(26) \cap \<7, 11, 12, 17\> \cap \<9, 10, 13, 16, 17, 21\>,
$$
demonstrates the decompositions constructed in Theorem~\ref{t:ordinarydecomp} need not achieve the minimum possible decomposition length.  In fact, one can verify computationally that $m = 28$ is the smallest multiplicity for which this phenomenon occurs, and 
$$
H_{56} = I(55) \cap I(54) \cap \<8, 15, 19, 41\> \cap \<7, 15, 23, 31, 39, 47\>
$$
is the first example where the shortest decomposition in Theorem~\ref{t:ordinarydecomp} exceeds the minimum possible length by at least 2.  Note that by \cite[Theorem~2.5]{b-f}, the minimum decomposition length of $H_m$ is unbounded as $m \to \infty$.  
\end{remark}




\end{document}